\documentclass[a4,11pt, leqno]{article}
\usepackage{amsmath,amsfonts,amsthm,latexsym,amssymb,mathrsfs}
\def\H_0{\mathcal{H}_0(T)}

\def\X{{\mathcal X}}

\def\ST{\tau_{AB}}

\def\X{{\cal X}}
\def\Y{{\cal Y}}
\def\H{{\cal H}}

\newtheorem{df}{Definition}[section]
\newtheorem{thm}[df]{Theorem}
\newtheorem{pro}[df]{Proposition}
\newtheorem{cor}[df]{Corollary}

\newtheorem{rema}[df] {Remark}
\newtheorem{lem}[df] {Lemma}
\def\sfstp{{\hskip-1em}{\bf.}{\hskip1em}}

\def\enddemo{\qed \endtrivlist}
\expandafter\let\csname enddemo*\endcsname=\enddemo

\def\qedsymbol{\ifmmode\bgroup\else$\bgroup\aftergroup$\fi
  \vcenter{\hrule\hbox{\vrule
height.6em\kern.6em\vrule}\hrule}\egroup}
\def\qed{\ifmmode\else\unskip\nobreak\fi\quad\qedsymbol}

\begin{document}
\title
{ \bf Generalized Browder's theorem for \\tensor product
and elementary operators\/}

\author {\normalsize Enrico   Boasso, B. P.  Duggal }

\date{ }

%%%%%%%%%%%%%%%%%%%%%%%%%%%%%%%%%%%%%%%%%%%%%%%%%%%%%%%%%%%%%%%%%%%%%%%%%%%%%%%%%%%%
%%%%%%%%%%%%%%%%%%%%%%%%%%%%%%%%%%%%%%%%%%%%%%%%%%%%%%%%%%%%%%%%%%%%%%%%%%%%%%%%%%%%

\maketitle \thispagestyle{empty} %\baselineskip=14pt

%\subject{enrico\_odisseo@yahoo.it (E. Boasso),bpduggal@yahoo.co.uk (B. P. Duggal).}

\vskip 1truecm

\setlength{\baselineskip}{12pt}

%%%%%%%%%%%%%%%%%%%%%%%%%%%%%%%%%%%%%%%%%%%%%%%%%%%%%%%%%%
\begin{abstract} \noindent The transfer property for the generalized Browder's theorem both of the tensor product and
of the left-right multiplication operator will be characterized in
terms of the $B$-Weyl spectrum inclusion. In addition, the isolated
points of these two classes of operators will be fully
characterized.\par

\end{abstract}

%%%%%%%%%%%%%%%%%%%%%%%%%%%%%%%%%%%%%%%%%%%%%%%%%%%%%%%%%%%%%%%%

\section {\sfstp Introduction}\setcounter{df}{0}
\ \indent In the recent past the relationship between, on the one
hand, Weyl and Browder's theorems and their generalizations and, on
the other, tensor products and elementary operators has been
intensively studied, see for example \cite{SK, AD, HK, KD, D2, DDK,
DHK, BDJ}. In particular,  given two operators that satisfy
Browder's theorem, it is proved in \cite{DDK} that a necessary and
sufficient condition   for the tensor product operator to satisfy Browder's
theorem is that the Weyl spectrum identity holds, see the latter
cited article or section  4.\par

\indent The main objective of this work is to characterize when
given two operators that satisfy the generalized Browder's theorem, the
tensor product operator also satisfies the generalized Browder's
theorem, using in particular the $B$-Weyl spectrum identity.
Furthermore, since one inclusion always holds for operators
satisfying the generalized Browder's theorem, it is enough to consider
the $B$-Weyl spectrum inclusion, see section 4. It is worth noticing
that since Browder's and the generalized Browder's theorem are
equivalent (\cite{AZ}), the results of this work also provide a
characterization for the transfer property of the Browder's theorem
for the tensor product operator.\par

\indent However, to prove the key characterization of section 4, the
set of isolated points of the tensor product operator need to be
studied. In particular, after section 2 where several basic
definitions and facts will be recalled,  the poles and the
complement of the poles in the isolated points of the tensor product
operator will be characterized in terms of the corresponding sets of
the source operators. It is important to note that these results
continue and deepen the characterization of the isolated points of
the tensor product operator presented in \cite{HK}, see section
3.\par \indent Finally, since the  same arguments can be applied to
the left-right multiplication operator, similar characterizations
will be proved for elementary operators.
\newpage

\section {\sfstp Preliminary definitions}\setcounter{df}{0}
 \hskip.5truecm

\indent From now on $\X$ and $\Y$ shall denote infinite dimensional
complex Banach spaces and $B(\X, \Y)$ the algebra of all bounded
linear maps defined on $\X$ and with values in $\Y$. As usual, when
$\X=\Y$, $B(\X,\X)=B(\X)$. Given $A\in B(\X)$, $N(A)$, $R(A)$,
$\sigma (A)$ and $\sigma_a(A)$ will stand for the null space, the
range, the spectrum and the approximate point spectrum of $A$
respectively. In addition, $\X^*$ will denote the dual space of
$\X$, and if $A\in \X$, then $A^*\in B(\X^*)$ will stand for the
adjoint map of $A$. \par

\indent Recall that $A\in B(\X)$  is said to be  a
\it Weyl  operator\rm, if the dimensions both of $N(A)$ and of $\X/R(A)$
are finite and equal. Let $\sigma_w(A)$ be the \it
Weyl spectrum of $A$\rm, i.e.,
$\sigma_w(A)=\{ \lambda\in\mathbb C\colon A-\lambda\hbox{ is not Weyl}\}$,
where $A-\lambda$ stands for $A-\lambda I$, $I$ the identity map of $\X$.
Note, in addition, that the concept of Weyl operator has been generalized recently.
 An operator $A\in B(\X)$ will be said to be \it
B-Weyl\rm, if there exists $n\in\mathbb
N$ for which the range of $R(A^n)$ is closed and the induced
operator $A_n\in B(R(A^n))$ is Weyl (\cite{B3}). It is worth noticing that if for some $n\in\mathbb N$, $A_n\in
B(R(A^n))$ is Weyl, then $A_m\in B(R(A^m))$ is Weyl for
all $m\ge n$ (\cite{B1}). Naturally, from this class of
operators the B-Weyl spectrum of $A\in B(\X)$ can be derived in the usual way;
this spectrum will be denoted by
$\sigma_ {BW}(A)$. \par

\indent On the other hand, a Banach space operator $A\in B(\X)$ is said to be \it Drazin invertible\rm,
if there exists a necessarily unique $B\in B(\X)$ and some $m\in \mathbb N$ such that
$$
A^m=A^mBA, \hskip.3truecm BAB=B, \hskip.3truecm AB=BA.
$$
 If $DR(B(\X))=\{ A\in B(\X)\colon
A\hbox{ is Drazin invertible} \}$, then the \it Drazin spectrum \rm of
$A\in B(\X)$ is the set $\sigma_{DR}(A)=\{\lambda\in \mathbb C\colon
A-\lambda\notin DR(B(\X) )\}$ (\cite{BS, Bo}). \par

\indent The \it ascent \rm  (respectively \it the descent\rm ) of
$A\in B(\X)$ is the smallest non-negative integer $a=asc (A)$
(respectively $d=dsc (A)$) such that $N(A^a)=N(A^{a+1})$
(respectively $R(A^d)=R(A^{d+1})$); if such an integer does not
exist, then $asc(A)=\infty$ (respectively $dsc(A)=\infty$). Recall
that $\lambda\in \sigma (A)$ is said to be \it a pole \rm of $A$, if
the ascent and the descent of $A-\lambda$ are finite (hence equal).
The set of poles of $A\in B(\X)$ will  be denoted by $\Pi (A)$. Note
that $\Pi (A)=\sigma (A)\setminus \sigma_{DR}(A)$ (\cite[Theorem
4]{K}). In particular, if $A\in B(\X)$ is quasi-nilpotent, then
according to \cite[Theorem 5]{K}, necessary and sufficient for $A$
to be nilpotent is that $\Pi(A)=\{0\}$. In addition, the set of \it
poles of finite rank \rm of $A$ is the set $\Pi_0(A)=\{\lambda\in
\Pi(A)\colon \alpha(A-\lambda)<\infty\}$, where $\alpha(A-\lambda)=
\rm{dim}N(A-\lambda)$.

\indent Recall that an operator $A\in B(\X)$ is said  to satisfy \it Browder's theorem\rm,
 if $\sigma_{w}(A)=\sigma(A)\setminus \Pi_0(A)$, while $A$ is said to satisfy the
\it generalized Browder's theorem\rm, if  $\sigma_{BW}(A)=\sigma(A)\setminus \Pi(A)=\sigma_{DR}(A)$.
According to \cite[Theorem 2.1]{AZ}, the Browder's and the generalized Browder's theorems are equivalent.
Moreover, according to \cite[Theorem 2.1(iv)]{CH}, the generalized Browder's theorem is equivalent to
the fact that acc $\sigma(A)\subseteq \sigma_{BW} (A)$. Here and elsewhere in this article,
for $K\subseteq  \mathbb C$, iso $K$ will stand for the set of isolated points of
$K$ and acc $K=K\setminus$ iso $K$ for the set of limit points of $K$. The generalized Browder's theorem was
studied in \cite{A, AZ, CH, D, BDJ}.\par

\indent In what follows, given Banach spaces $\X$ and $\Y$, $\X\overline{\otimes}\Y$ will stand for
the completion, endowed with a reasonable uniform cross-norm, of the
algebraic tensor product $\X\otimes \Y$ of $\X$ and $\Y$. In addition,
if $A\in B(\X)$ and $B\in B(\Y)$, then $A\otimes B\in B(\X\overline{\otimes}\Y)$ will denote
the tensor product operator defined by $A$ and $B$.\par

\indent On the other hand, $\ST\in B(B(\Y, \X))$ will denote the
multiplication operator defined by $A\in B(\X)$ and $B\in B(\Y)$,
i.e., $\ST (U)=AUB$, where $U\in B(\Y, \X)$ and $\X$ and $\Y$ are
two Banach spaces. Note that $\ST=L_AR_B$, where $L_A\in  B(B(\Y,
\X))$ and $R_B\in  B(B(\Y, \X))$ are the left and right
multiplication operators defined by $A$ and $B$ respectively,
i.e., $L_A(U)=AU$ and $R_B(U)=UB$, $U\in  B(\Y, \X)$.\par

 \section {\sfstp The isolated points }\setcounter{df}{0}
\
\indent In this section the isolated points both of the tensor product and of the left-right
multiplication operator will be studied. To this end, some preparation is needed.\par

\begin{rema}\label{rem7} \rm Let $\X$ be a Banach space, consider $A\in B(\X)$ and
set $I(A)=$ iso $\sigma (A)\setminus \Pi(A)$. \par
\noindent (i) Necessary and sufficient for $\lambda\in \sigma(A)$ to belong to $I(A)$ is that
there exist $M$ and $N$, two closed and complemented subspaces of $\X$ invariant for $A$, such that
if $A_1=A\mid_M$ and $A_2=A\mid_N$, then $A_1-\lambda$ is quasi-nilpotent but not nilpotent and $A_2-\lambda$ is invertible.
Note that $\sigma (A)= I(A)=\{\lambda\}$ if and only if $N=0$.\par

\noindent (ii)  Let $\lambda\in \sigma(A)$. The complex number $\lambda$ belongs to
$\Pi(A)$ if and only if there are $M'$ and $N'$ two closed and complemented
subspaces of $\X$ invariant for $A$, such that if
$A'= A\mid_{M'}$
and $A''=A\mid_{N'}$, then $A'-\lambda$
is nilpotent and $A''-\lambda$ is invertible. As in statement (i), $\sigma (A)=\Pi (A)=\{\lambda\}$ is equivalent to the
fact that $N'=0$.\par
\noindent Statements (i)-(ii) are well known and they can be easily deduced  from
\cite[Theorem 12]{Bo} and \cite[Theorem 5]{K}. Now let $\Y$ be a Banach space and consider $B\in B(\Y)$.\par
\noindent (iii) Since $\sigma (A\otimes B)=\sigma(A)\sigma(B)=\sigma (\ST)$ (\cite[Theorem 2.1]{I} and \cite[Corollary 3.4]{E}),
according to \cite[Theorem 6]{HK},
$$
(\hbox{iso }\sigma(A\otimes B))\setminus \{ 0\}= (\hbox{iso } (\ST))\setminus \{ 0\}=
(\hbox{iso } \sigma(A)\setminus \{ 0\})( \hbox{iso } \sigma(B)\setminus \{ 0\}).
$$
\noindent (iv) Set
$$
\mathbb L= (I(A)\setminus \{ 0\}) (I(B)\setminus \{ 0\})\cup  (I(A)\setminus \{ 0\}) (\Pi (B)\setminus \{ 0\})
\cup  (\Pi (A)\setminus \{ 0\}) (I(B)\setminus \{ 0\}).
$$
Then clearly, (iso $\sigma(A\otimes B))\setminus \{ 0\}=(\hbox{iso }
(\ST))\setminus \{ 0\}= \mathbb L\cup (\Pi (A)\setminus \{ 0\})(\Pi
(B)\setminus \{ 0\})$.\par

\noindent (v) Let $\lambda\in$ (iso
$\sigma (A\otimes B))\setminus\{0\}=(\hbox{iso } (\ST))\setminus \{
0\}$. Then, it is not difficult to prove that there exist finite
sequences $\{\mu_i\}_{i=1}^n$ and $\{\nu_i\}_{i=1}^n$ of points
$\mu_i\in$ iso $\sigma(A)\setminus \{ 0\}$ and $\nu_i\in$ iso
$\sigma(B)\setminus \{ 0\}$ such that $\lambda=\mu_i \nu_i$ for all
$i=1,\dots , n$.\par

\noindent (vi) Note that  if $0\in $ iso $\sigma (A\otimes B)=$
iso $\sigma (\ST)$, then one of the following possibilities holds:\par
\noindent (a) $\sigma(A)=\{0\}$ or $\sigma(B)=\{0\}$;\par
\noindent (b) ($\sigma(A)\neq\{0\}$ and $\sigma(B)\neq\{0\}$)
$0\in $ iso $\sigma(A)$ and $0\notin\sigma (B)$ or
$0\notin\sigma(A)$ and $0\in$ iso $\sigma(B)$;\par
\noindent (c) ($\sigma(A)\neq\{0\}$, $\sigma(B)\neq\{0\}$, $0\in \sigma(A)\cap\sigma(B)$)
 $0\in$ iso $\sigma(A)\cap$ iso $\sigma(B)$.
\end{rema}

\indent In the next theorem the position of $0\in\mathbb C$ in the isolated points will be characterized.
To this end, if $\X$ and $\Y$ are two Banach spaces, then $I_1$ and $I_2$ will denote the identity map on $\X$ and $\Y$ respectively.
Moreover, given $x\in \X$ and $f\in \Y^*$, $U_{x,f}\in B(\Y,\X)$ is the map defined as follows:
$U_{x,f}(y)=xf(y)$, $y\in \Y$.\par

\begin{thm}\label{thm8} Let $\X$ and $\Y$ be two Banach spaces and consider
$A\in B(\X)$, $B\in B(\Y)$,  $A\otimes B\in B(\X\overline{\otimes}\Y)$ and $\ST\in B(B(\Y,\X))$.
Suppose that $0\in \hbox{\rm iso }\sigma(A\otimes B)= \hbox{\rm iso }\sigma (\ST)$.\par
\noindent \rm (i)\it If $\sigma(A)=\Pi(A)=\{0\}$ or $\sigma(B)=\Pi(B)=\{0\}$, then $\sigma(A\otimes B)=\Pi (A\otimes B)=\{0\}=\Pi(\ST)=\sigma(\ST)$.\par
\noindent \rm (ii) If  $\sigma(A)=I(A)=\{0\}$ and $B$ is not nilpotent or $\sigma(B)=I(B)=\{0\}$ and $A$ is not nilpotent,
then $\sigma (A\otimes B)=I(A\otimes B)=\{0\}=I(\ST)=\sigma(\ST)$.\par
\noindent \rm (iii)\it If $0\in \Pi(A)$ and $0\notin \sigma(B)$ or $0\notin \sigma(A)$ and $0\in\Pi(B)$,
 then $0\in\Pi (A\otimes B)\cap\Pi(\ST)$.\par
\noindent \rm (iv)\it If $0\in I(A)$ $(\sigma (A)\neq \{0\})$ and $0\notin \sigma(B)$ or $0\notin \sigma(A)$ and $0\in I(B)$ $(\sigma (B)\neq\{0\})$,
 then $0\in I(A\otimes B)\cap I(\ST)$.\par
\noindent \rm (v)\it If $0\in \Pi (A)\cap \Pi (B)$, then $0\in \Pi (A\otimes B)\cap\Pi (\ST)$.\par
\noindent \rm (vi)\it If $0\in I(A)\cap \Pi (B)$ and $B$ is not nilpotent, $0\in \Pi (A)\cap I (B)$ and  $A$ is not nilpotent,
or $0\in  I(A)\cap I (B)$, then $0\in I(A\otimes B)\cap I(\ST)$.\par
\end{thm}
\begin{proof} (i). According to Remark \ref{rem7}(ii), $A$ or $B$ is nilpotent, which implies that $A\otimes B$
is nilpotent.\par
\indent On the other hand, since $L_A\in B(B(\X))$ or $R_B\in B(B( \Y))$ is nilpotent, $\ST$ is nilpotent.\par

\noindent (ii). Suppose that $\sigma(A)=I(A)=\{0\}$ and $B$ is not nilpotent. Clearly, $\sigma (A\otimes B)=\{0\}$.
In addition, according to Remark \ref{rem7}(i), $A$ is not  nilpotent.  In particular,
for each $k\in \mathbb N$ there exist $x_k\in \X$ and $y_k\in \Y$
such that $\parallel A^k(x_k)\parallel=1$ and $\parallel B^k(y_k)\parallel=1$.
Therefore, since $\X\overline{\otimes}\Y$ is endowed with a reasonable uniform cross norm,
 $\parallel (A\otimes B)^k (x_k\otimes y_k)\parallel =1$, for each $k\in\mathbb N$.
As a result, $A\otimes B$ is not nilpotent, equivalently $I(A\otimes B)=\{0\}$. \par

\indent On the other hand, it is clear that $\sigma(\ST)=\{0\}$. Moreover, since $B$ is not nilpotent,
$B^*\in B(\Y^*)$ is not nilpotent. In particular, for each $k\in \mathbb N$ there exist
$x_k\in \X$ and $f_k\in \Y^*$ such that $\parallel A^k(x_k)\parallel=1$ and $\parallel (B^*)^k(f_k)\parallel=1$.
Consider $U_{x_k,f_k}\in B(\Y,\X)$. Then, $\parallel\ST^k(U_{x_k,f_k})\parallel=1$. Consequently, $\ST$ is not
nilpotent and $I(\ST)=\{0\}$.\par

\indent The remaining case can be proved in a similar way.\par

\noindent (iii). If $0\in \Pi(A)$ and $0\notin \sigma(B)$ or $0\notin \sigma(A)$ and $0\in\Pi(B)$,
then it is not difficult to prove that $A\otimes B$ and $\ST$ are Drazin invertible,
equivalently $0\in\Pi (A\otimes B)\cap\Pi(\ST)$.\par

\noindent (iv)  If $0\in I(A)$, then, according to Remark \ref{rem7}(i), there exist $M_1$ and $M_2$
two closed and complemented subspaces of $\X$ invariant for $A$ such that
$A_1\in B(M_1)$ is quasi-nilpotent but not nilpotent and $A_2\in B(M_2)$ is invertible,
where $A_1=A\mid_{M_1}$ and $A_2=A\mid_{M_2}$. Now, clearly
$\X\overline{\otimes}\Y= M_1\overline{\otimes} \Y\oplus M_2\overline{\otimes} \Y$,
$A_1\otimes B$ is quasi-nilpotent and, since $M_2\neq 0$ $(\sigma (A)\neq\{0\})$, $A_2\otimes B$
is invertible. However, using an argument similar to the one in the proof of statement (ii), $A_1\otimes B$
is not nilpotent. Consequently, according to Remark  \ref{rem7}(i), $0\in I(A\otimes B)$.\par

\indent To prove that $0\in I(\ST)$, consider the decompositions of $X$ and $A$ recalled in the previous
paragraph. Note that $B(\Y,\X)=B(\Y,M_1)\oplus B(\Y, M_2)$ and then, decomposing $\ST$
as a block operator, $\ST$ is a diagonal operator with entries $\tau_{A_1B}\in B( B(\Y,M_1))$
and $\tau_{A_2B}\in B( B(\Y,M_2))$. Clearly,  $\tau_{A_1B}$ is quasi-nilpotent and  $\tau_{A_2B}$
is invertible. However, using an argument similar to the one in the proof of statement (ii),
$\tau_{A_1B}$ is not nilpotent. In particular, $0\in I(\ST)$.\par

\indent The remaining case  can be proved in a similar way.\par

\noindent (v). If $0\in \Pi (A)\cap \Pi (B)$, then $A$ and $B$ are Drazin invertible,
which implies that $A\otimes I_2$ and $I_1\otimes B$ are Drazin invertible.
Since  $A\otimes I_2$ and $I_1\otimes B$  commute, according to \cite[Proposition 2.6]{BS},
$A\otimes B$ is Drazin invertible, equivalently $0\in \Pi (A\otimes B)$.\par

\indent On the other hand, it is not difficult to prove that $L_A$ and $R_B$ are
Drazin invertible. Moreover, since $L_A$ and $R_B$ commute, $\ST$ is Drazin invertible,
in particular $0\in \Pi(\ST)$.\par

\noindent (vi). If $0\in I(A)\cap \Pi (B)$, then, according to Remark \ref{rem7}(i)-(ii),
there exist $M_1$ and $M_2$ (respectively
$N_1$ and $N_2$) two closed and complemented subspaces
of $X$ (respectively $Y$) invariant for $A$ (respectively $B$)
such that $A_1$ is quasi-nilpotent but not nilpotent and $A_2$
is invertible (respectively $B_1$ is nilpotent and $B_2$ is invertible),
where $A_1=A\mid_{M_1}$ and $A_2=A\mid_{M_2}$
(respectively $B_1=B\mid_{N_1}$ and $B_2=B\mid_{N_2}$).
Now, it is clear that $\X\overline{\otimes}\Y= M_1\overline{\otimes}N_1\oplus M_2\overline{\otimes} N_1\oplus M_1\overline{\otimes}N_2
\oplus M_2\overline{\otimes} N_2$, $A_1\otimes B_1\in B(M_1\overline{\otimes}N_1)$ and
$A_2\otimes B_1\in B(M_2\overline{\otimes}N_1)$ are nilpotent,
$A_1\otimes B_2\in B(M_1\overline{\otimes}N_2)$ is quasi-nilpotent and $A_2\otimes B_2\in
B(M_2\overline{\otimes} N_2)$ is invertible. As a result,
to prove that $0\in I(A\otimes B)$, it is enough to prove that $A_1\otimes B_2\in B(M_1\overline{\otimes}N_2)$
is not nilpotent. However, since $B$ is not nilpotent, $N_2\neq 0$, and then, using the argument
in the proof of statement (ii), $A_1\otimes B_2$ is not nilpotent.  \par

\indent On the other hand, according to the decomposition of $\X$ and $\Y$ recalled in the previous paragraph,
$\ST\in B(B(\Y,\X))$ can be considered as a diagonal  operator with   diagonal entries $(\ST)_{11}\in B(B(N_1,M_1))$,
$(\ST)_{22}\in B(B(N_2,M_1))$, $(\ST)_{33}\in B(B(N_1,M_2))$ and $(\ST)_{44}\in B(B(N_2,M_2))$.
Clearly, $(\ST)_{11}$ and $(\ST)_{33}$ are nilpotent, $(\ST)_{44}$ is invertible and $(\ST)_{22}$ is quasi-nilpotent.
Thus, to prove that $0\in I(\ST)$, it is enough to prove that $(\ST)_{22}$ is not nilpotent.
However, since $N_2\neq 0$ and $B_2$ is invertible, using the argument  in the proof
of statement (ii), $(\ST)_{22}\in B(B(N_2,M_1))$
is not nilpotent. \par

\indent Similar arguments prove the remaining cases both for $A\otimes B$ and for $\ST$.
\end{proof}

\indent The following proposition will be useful to study  the isolated non null-points.\par

\begin{pro}\label{thm28} Let $\X$ and $\Y$ be two Banach spaces and suppose that
$A\in B(\X)$ and $B\in B(\Y)$ are such that $\sigma (A)= \{\mu\}$, $\sigma(B)=\{\nu \}$,
$\mu\nu\neq 0$. Consider $A\otimes B\in B(\X\overline{\otimes}\Y)$ and $\ST\in B(B(\Y,\X))$.
Then, $\sigma (A\otimes B)=\sigma (\ST)=\{\mu\nu\}$ and the following statements hold.\par
\noindent \rm (i)\it If $A-\mu$ and $B-\nu$ are nilpotent, then
$A\otimes B-\mu\nu$ and $\ST-\mu\nu$ are nilpotent.\par
\noindent \rm (ii)\it If either $A-\mu$ or $B-\nu$ is quasi-nilpotent but not nilpotent,
then $A\otimes B-\mu\nu$ and $\ST-\mu\nu$ are not nilpotent.
\end{pro}
\begin{proof}Clearly $\sigma(A\otimes B)=\sigma (\ST)=\sigma(A)\sigma(B)=\{\mu\nu\}$. \par

\noindent (i). Note that
$A\otimes B-\mu\nu=(A-\mu)\otimes B+\mu\otimes (B-\nu)$.
Since $(A-\mu)\otimes B$ and $\mu\otimes (B-\nu)$ are nilpotent and commute,
an easy calculation proves that $A\otimes B-\mu\nu$ is nilpotent.\par
\indent On the other hand, since $\ST-\mu\nu= L_{(A-\mu)}R_B + \mu R_{(B-\nu)}$,
a similar argument proves that $\ST-\mu\nu$ is nilpotent.\par
\noindent (ii) Since $A\otimes B-\mu\nu=(A-\mu)\otimes B+\mu\otimes (B-\nu)$,
it is not difficult to prove that
$$
(A\otimes B-\mu\nu)I_1\otimes B^{-1}=I_1\otimes B^{-1}(A\otimes B-\mu\nu)=
(A-\mu)\otimes I_2 - \mu\nu\otimes (B^{-1}-\nu^{-1}).
$$
Moreover, since $A\otimes B-\mu\nu$ and $I_1\otimes B^{-1}$ commute,
$A\otimes B-\mu\nu$ is nilpotent if and only if $(A\otimes B-\mu\nu)I_1\otimes B^{-1}$ is nilpotent. \par

\indent Suppose that $(B-\nu)\in B(\Y)$ is quasi-nilpotent but not
nilpotent. Then, $(B^{-1}-\nu^{-1})\in B(\Y)$ is quasi-nilpotent but
not nilpotent. In fact, it is clear that $\sigma
(B^{-1})=\{\nu^{-1}\}$. In addition, if $B^{-1}-\nu^{-1}$ were
nilpotent, then a straightforward calculation proves that $B$ must
be algebraic. However, since $B-\nu$ is quasi-nilpotent, $B-\nu$ must
be nilpotent, which is impossible.\par

\indent Next note that since $\sigma_a(A)=\sigma (A)=\{\mu\}$, there exists
$(x_n)_{n\in\mathbb N}\subset \X$ such that $\parallel x_n\parallel=1$,
$n\in\mathbb N$, and $((A-\mu)(x_n))_{n\in\mathbb N}$ converges to $0\in\X$.
Then, given $k\in \mathbb N$, $c_{k,j}=\frac{k!}{(k-j)!j!}$ and $y_k\in \Y$ such that $\mid \mu\nu\mid^k \parallel (B^{-1}-\nu^{-1})^k(y_k)\parallel=2$, there exist $n_k\in \mathbb N$
such that for all $n\in\mathbb N$, $n\ge n_k$, $\parallel \sum_{j=1}^k c_{k,j} (-\mu\nu)^{k-j}(A-\mu)^j(x_n)\otimes (B^{-1}-\nu^{-1})^{k-j}(y_k)\parallel<1$.
As a result,  for $n\ge n_k$,
$$
\parallel ((A-\mu)\otimes I_2 - \mu\nu\otimes (B^{-1}-\nu^{-1}))^k(x_n\otimes y_k)\parallel > 1.
$$
Therefore,  $A\otimes B-\mu\nu$ is not nilpotent.  A similar
argument, using   $A\otimes B-\mu\nu= A\otimes (B-\nu)
+(A-\mu)\otimes \nu$, proves the case $A-\mu$ quasi-nilpotent but
not nilpotent for the tensor product operator.\par

\indent On the other hand, since $\ST-\mu\nu= L_{(A-\mu)}R_B + \mu R_{(B-\nu)}$, adapting the argument used before
it is not difficult to prove that
$\ST-\mu\nu$ is not nilpotent if and only if $L_{(A-\mu)}-\mu\nu R_{(B^{-1}-\nu^{-1})}$ is not nilpotent.
To prove this latter fact, consider the same sequence $(x_n)_{n\in\mathbb N}\subset \X$ of the tensor product
operator case. In addition, since $(B^{-1}-\nu^{-1})^*\in B(\Y^*)$ is not nilpotent,
for each $k\in\mathbb N$ there exists $f_k\in \Y^*$ such that
$\mid \mu\nu\mid^k \parallel ((B^{-1}-\nu^{-1})^*)^k(f_k)\parallel=2$. However,
an argument similar to the one used in  the tensor product operator case
proves that there is $n\in\mathbb N$ such that
$$
\parallel (L_{(A-\mu)} - \mu\nu\otimes R_{(B^{-1}-\nu^{-1})})^k(U_{x_n,f_k})\parallel > 1.
$$
Therefore,  $\ST-\mu\nu$ is not nilpotent.  A similar argument,
using  $\ST-\mu\nu= L_AR_{(B-\nu)} +\nu L_{(A-\mu)}$, proves the
case $A-\mu$ is quasi-nilpotent but not nilpotent for the
multiplication operator.
 \end{proof}

\indent Given $\X$ and $\Y$ two Banach spaces and $A\in B(\X)$
and $B\in B(\Y)$, in \cite[Theorem 6]{HK} the limit and the isolated points both of the tensor product operator
$A\otimes B\in B(\X\overline{\otimes}\Y)$ and of the elementary operator $\ST\in B(B(\Y,\X))$ were studied. In the following theorem
$I(A\otimes B)\setminus \{0\}$, $I(\ST)\setminus \{0\}$, $\Pi (A\otimes B)\setminus \{0\}$ and $\Pi(\ST)\setminus \{0\}$
will be characterized in terms of the corresponding sets of $A$ and $B$.
\par
\begin{thm}\label{thm9}Let $\X$ and $\Y$ be two Banach spaces and consider
$A\in B(\X)$ and $B\in B(\Y)$. Then, the following statements hold.\par
\noindent \rm (i) \it $\mathbb L= I(A\otimes B)\setminus \{0\}=I(\ST)\setminus \{0\}$.\par
\noindent \rm (ii) \it $\Pi (A\otimes B)\setminus \{0\}= \Pi(\ST)\setminus \{0\}=(\Pi (A)-\{0\})(\Pi (B)-\{0\})\setminus \mathbb L$.
\end{thm}
 \begin{proof} In the first place, note that according to Remark \ref{rem7}(iv),
statement (i) implies statement (ii).\par

\indent To prove statement (i), let $\lambda\in$ iso $\sigma
(A\otimes B)\setminus \{0\}$. Then, according to Remark
\ref{rem7}(v), there exist $n\in\mathbb N$ and finite spectral sets
$\{\mu\}=\{\mu_1,\ldots ,\mu_n\}\subseteq$ iso $\sigma (A)$ and
$\{\nu\}=\{\nu_1,\ldots ,\nu_n\}\subseteq$ iso $\sigma (B)$ such
that $\lambda=\mu_i\nu_i$ for all $1\le i\le n$. Corresponding to
these spectral sets there are closed subspaces $M_1$, $M_2$ and
$(M_{1i})_{i=1}^n$ of $\X$ invariant for $A$
 and closed subspaces $N_1$, $N_2$ and $(N_{1i})_{i=1}^n $ of $\Y$ invariant for $B$ such that
$\X=M_1\oplus M_2$, $M_1=\oplus_{i=1}^n M_{1i}$,  $\Y=N_1\oplus
N_2$, $N_1=\oplus_{i=1}^n N_{1i}$, $\sigma(A_1)=\{\mu\}$,
$\sigma(A_2)=\sigma(A)\setminus \{\mu\}$,
$\sigma(A_{1i})=\{\mu_i\}$, $\sigma(B_1)=\{\nu\}$,
$\sigma(B_2)=\sigma(B)\setminus \{\nu\}$ and
$\sigma(B_{1i})=\{\nu_i\}$, where $A_1=A\mid_{M_1}$,
$A_2=A\mid_{M_2}$, $A_{1i}=A\mid_{M_{1i}}$, $B_1=B\mid_{N_1}$,
$B_2=B\mid_{N_2}$ and $B_{1i}=B\mid_{N_{1i}}$. Note that $A\otimes
B-\lambda$ is invertible on the closed invariant subspaces
$M_1\overline{\otimes}N_2$, $M_2\overline{\otimes}N_1$,
$M_2\overline{\otimes}N_2$ and $M_{1j}\overline{\otimes}N_{1k}$,
$1\le j\neq k\le n$. Moreover, $\X\overline{\otimes}\Y$ is the
direct sum of these subspaces and $M_{1i}\overline{\otimes}N_{1i}$,
$1\le i\le n$.
\par

\indent Suppose that $\lambda\in\mathbb L$. Then, there exist $\mu\in$ iso $\sigma(A)\setminus \{0\}$
and $\nu\in$ iso $\sigma(B)\setminus \{0\}$ such that $\lambda=\mu\nu$ and either $\mu\in I(A)\setminus  \{0\}$
or $\nu\in I(B)\setminus  \{0\}$. Applying what has been done in the previous paragraph to $\lambda\in \mathbb L$,
there exist an $n=n(\lambda)\in \mathbb N$ and an $i$, $1\le i\le n$, such that
$\mu=\mu_i$ and $\nu=\nu_i$. Therefore, according to Proposition \ref{thm28}(ii) and Remark
 \ref{rem7}(i), $\lambda\in I(A\otimes B)\setminus  \{0\}$.\par

\indent On the other hand, consider  $\lambda\in I(A\otimes B)\setminus  \{0\}$.
As before,
there exist an $n=n(\lambda)\in \mathbb N$ and  $\mu_i\in$ iso $\sigma(A)\setminus \{0\}$
and $\nu_i\in$ iso $\sigma(B)\setminus \{0\}$ such that $\lambda=\mu_i\nu_i$, $i=1,\ldots ,n$.
Now, if $\lambda\notin \mathbb L$, then for each $i=1,\ldots ,n$, $\mu_i\in (\Pi(A)\setminus \{0\})$ and $\nu_i\in (\Pi (B)\setminus \{0\})$.
However, according to Proposition \ref{thm28}(i) and Remark  \ref{rem7}(ii), $\lambda\in (\Pi (A\otimes B)\setminus  \{0\})$,
which is impossible. \par

\indent To prove that  $\mathbb L= I(\ST)\setminus \{0\}$, as in the
tensor product operator case, consider the decompositions of $\X$
and $\Y$ into closed complemented invariant subspaces for $A$ and
$B$ respectively and, as in  Theorem \ref{thm8}, decompose $\ST$ as
a diagonal operator. Then, to conclude the proof, adapt the argument
developed to prove that $\mathbb L= I(A\otimes B)\setminus \{0\}$ to
the case under consideration.
\end{proof}

\indent Applying the main results of this section, it is not difficult to prove that the Drazin spectra of the
tensor product and of  the elementary operator coincide. Note that since the spectra of these operators are equal,
both the set of limit points and  the one of isolated points of the aforementioned operators are identical.\par

\begin{cor}\label{cor48} Let $\X$ and $\Y$ be two Banach spaces and consider
$A\in B(\X)$ and $B\in B(\Y)$. Then, the following statements hold.\par
\noindent \rm (i)\it  $\Pi(A\otimes B)=\Pi(\ST)$.\par
\noindent \rm (ii)\it $I(A\otimes B)=I(\ST)$.\par
\noindent \rm (iii)\it $\sigma_{DR}(A\otimes B)=\sigma_{DR}(\ST)$.
\end{cor}
\begin{proof} Statements (i)-(ii) can be derived from Theorems \ref{thm8} and \ref{thm9}.
To prove statement (iii), apply \cite[Theorem12]{Bo}.
\end{proof}

 \section {\sfstp The B-Weyl spectrum inclusion}\setcounter{df}{0}

\

\indent Recall that given $A\in B(\X)$ and $B\in B(\Y)$ two
operators satisfying Browder's theorem, the  \it Weyl spectrum
equality for $A\otimes B$\rm, i.e., the identity
$$
\sigma_w(A\otimes B)= \sigma (A)\sigma_w(B)\cup\sigma_w (A)\sigma(B),
$$
is equivalent to the the fact that $A\otimes B$
satisfies Browder's theorem (\cite[Theorem 3]{DDK}).
Note that the inclusion
$$
\sigma_w(A\otimes B)\subseteq \sigma (A)\sigma_w(B)\cup\sigma_w (A)\sigma(B)
$$
always  holds, so that the relevant inclusion is the reverse
inclusion ``$\supseteq$".
\par

\indent Similarly, under the same conditions for $A$ and $B$, the \it Weyl spectrum equality for $\ST$\rm, i.e., the identity
$$
\sigma_w(\ST)= \sigma (A)\sigma_w(B)\cup\sigma_w (A)\sigma(B),
$$
is equivalent to the the fact that $\ST$
satisfies Browder's theorem (\cite[Theorem 4.5]{BDJ}). As in the tensor product operator case,
the following inclusion always holds:
$$
\sigma_w(\ST)\subseteq \sigma (A)\sigma_w(B)\cup\sigma_w (A)\sigma(B).
$$

\indent Given  $A\in B(\X)$ and $B\in B(\Y)$ two operators that
satisfy the generalized Browder's theorem, the \it B-Weyl spectrum inclusion for $A\otimes B$ \rm (respectively \it for $\ST$\rm)
will be said to hold, if
$$
\sigma (A)\sigma_{BW}(B)\cup \sigma_{BW}(A)\sigma(B)\subseteq \sigma_{BW}(A\otimes B)
$$
$$
\hbox{\rm (respectively if }\sigma (A)\sigma_{BW}(B)\cup \sigma_{BW}(A)\sigma(B)\subseteq \sigma_{BW}(\ST)).
$$
\indent In this section the B-Weyl spectrum inclusion will be
studied in relation to the  \it transfer property for the
generalized Browder's theorem\rm, i.e., the conditions under which
given $A\in B(\X)$ and $B\in B(\Y)$ two operators that satisfy the\
generalized Browder's theorem, $A\otimes B\in  B(\X \otimes \Y)$ and
$\ST\in B(B(\Y,\X))$ also satisfy the generalized Browder's theorem.
Note that since the Browder's and generalized Browder's theorems are
equivalent (\cite[Theorem 2.1]{AZ}), the results of this section
also provide a characterization of the transfer property for the
Browder's theorem both for the tensor product and the left-right
multiplication operator.\par

\indent In the first place the B-Weyl spectrum  inclusion will be proved
to be an equality, when it holds. However, since for the main
results of this article the relevant condition is an inclusion,  the
B-Weyl spectrum inclusion will be focused on.\par

\begin{lem}\label{lem16} Let $\X$ and $\Y$ be two Banach spaces and consider $A\in B(\X)$ and
$B\in B(\Y)$ two operators that satisfy the generalized Browder's theorem. Then,
$$
  (\sigma_{BW}(A\otimes B)\cup\sigma_{BW}(\ST))\subseteq\sigma(A)\sigma_{BW}(B)\cup\sigma_{BW}(A)\sigma (B).
$$
\end{lem}
\begin{proof} Suppose that $0\in \sigma_{BW}(A\otimes B)$. Then, according
to \cite[Theorem 2.3]{B3}, $0\in $ acc $\sigma(A\otimes B)$ or $0\in I(A\otimes B)$.
Since $A$ and $B$ satisfy the generalized Browder's theorem, if $0\in $ acc $\sigma(A\otimes B)$, then
$0\in $ acc $\sigma(A)\subseteq \sigma_{BW}(A)$ or $0\in $ acc $\sigma(B)\subseteq \sigma_{BW}(B)$.
On the other hand, if $0\in I(A\otimes B)$, then according to Theorem \ref{thm8}, $A$ and
$B$ are not nilpotent and $0\in I(A)\subseteq  \sigma_{BW}(A)$ or $0\in I(B)\subseteq  \sigma_{BW}(B)$.
However, in all these cases $0\in \sigma(A)\sigma_{BW}(B)\cup\sigma_{BW}(A)\sigma (B)$.\par

Next consider $\lambda\neq 0$, $\lambda\in \sigma(A\otimes
B)\setminus (\sigma(A)\sigma_{BW}(B)\cup\sigma_{BW}(A)\sigma (B))$.
In particular,  for each $\mu\in \sigma(A)$ and $\nu\in \sigma(B)$
such that $\lambda=\mu\nu$, $\mu\in\sigma(A)\setminus\sigma_{BW}(A)$
and $\nu\in\sigma(B)\setminus\sigma_{BW}(B)$. However, since $A$ and
$B$ satisfy the generalized Browder's theorem, $\mu \in \Pi(A)$ and
$\nu\in \Pi(B)$. Consequently, according to Theorem \ref{thm9},
$\lambda\in\Pi (A\otimes B)$. Therefore, $\lambda\in\sigma(A\otimes
B)\setminus\sigma_{BW}(A\otimes B)$ (\cite[Theorem 2.3]{B3}).\par

\indent A similar argument proves the inclusion for $\ST$.
\end{proof}
\indent In what follows the transfer property for the generalized Browder's theorem will be studied.\par

\begin{thm}\label{thm5} Let $\X$ and $\Y$ be two Banach spaces and consider $A\in
B(\X)$ and $B\in B(\Y)$ two operators that satisfy the generalized Browder's theorem.
If the B-Weyl spectrum inclusion  for $A\otimes B$ (respectively for $\ST$) holds, then
$A\otimes B$ (respectively $\ST$) satisfies the generalized Browder's theorem.
\end{thm}

\begin{proof} According to \cite[Theorem 2.1(iv)]{CH}, acc $\sigma(A)\subseteq \sigma_{BW} (A)$
and  acc $\sigma(B)\subseteq \sigma_{BW} (B)$. Now, since the B-Weyl spectrum
inclusion for $A\otimes B$ holds, according to  \cite[Theorem 6]{HK},
$$
\hbox{\rm acc }\sigma(A\otimes B)\subseteq\sigma (A)(\hbox{\rm acc }\sigma (B))\cup(\hbox{\rm acc }\sigma (A))\sigma (B)\subseteq
\sigma (A)\sigma_{BW}(B)\cup \sigma_{BW}(A)\sigma(B)\subseteq \sigma_{BW}(A\otimes B).
$$
\indent Therefore,  $A\otimes B$ satisfies the generalized Browder's theorem.
Since $\sigma(\ST)=\sigma(A\otimes B)=\sigma(A)\sigma(B)$, the same argument proves
the statement concerning the operator $\ST$.
\end{proof}

\begin{rema}\label{rem6} \rm (i) Note that the converse of Theorem \ref{thm5} does not
in general hold. In fact, let $\X$ and $\Y$ be two Banach spaces and consider
$A\in B(\X)$ and $B\in B(\Y)$ two operators such that $A$ is nilpotent and $B$
satisfies the generalized Browder's theorem. As a result, $A\otimes B\in B(\X\overline{\otimes} \Y)$ is nilpotent,
what is more, $A$ and $A\otimes B$ satisfy the generalized Browder's theorem (the sets of limit points
of these two operators are empty). On the other hand,  since $A$ and $A\otimes B$ are nilpotent,  according to \cite[Theorem 2.3]{B3},
$\sigma_{BW} ( A)=\emptyset=\sigma_{BW} (A\otimes B)$. In particular, necessary
and sufficient for  $\sigma (A)\sigma_{BW}(B)\cup\sigma_{BW}(A)\sigma(B)= \emptyset$
is that $\sigma_{BW}(B)=\emptyset$ (observe, however, that the operators $A$, $B$ and $A\otimes B$
satisfy the equality $\sigma_w(A\otimes B)= \sigma(A)\sigma_w(B)\cup \sigma_w(A)\sigma(B)$). Naturally, the same can be said
for the operator $\ST$.\par

\noindent (ii) Let $\X$ be a Banach space and consider $A\in B(\X)$
an operator that  satisfies the generalized Browder's theorem. According
to \cite[Theorem 3]{Bo}, \cite[Theorem 1.5]{BKMO} and \cite[Theorem
2.7]{BBO}, $\sigma_{BW}(A)=\emptyset$ if and only if $A$ is
algebraic, i.e., there exists  a non-constant polynomial
$P\in\mathbb C[X]$ such that $P(A)=0$. Clearly, since the spectrum
of an algebraic operator is a finite set (actually in this case
$\sigma (A)=\Pi (A)$ (\cite[Theorem 1.5]{BKMO})), algebraic
operators satisfy the generalized Browder's theorem. Moreover, if $A\in
B(\X)$ and $B\in B(\Y)$ are algebraic operators, then
$\sigma(A)\sigma_{BW}(B)\cup\sigma_{BW}(A)\sigma (B)=\emptyset$ and
the B-Weyl spectrum inclusion both for $A\otimes B\in  B(
\X\overline{\otimes}\Y)$ and for $\ST\in  B(B(\Y,\X))$ holds.
Furthermore, it is not difficult to prove, using in particular
Theorems \ref{thm8} and \ref{thm9}, that $A\otimes B$ and $\ST$
satisfy the generalized Browder's theorem and $\sigma_{BW}(A\otimes
B)=\emptyset= \sigma_{BW}(\ST)$. Therefore, to characterize when the
transfer property implies the B-Weyl spectrum inclusion, it is
enough to consider two cases: first, when only one operator is
algebraic (observe that according to (i) the algebraic operator must
not be nilpotent); second, when both operators are not algebraic.
\end{rema}

\indent Before going on, to study the converse of Theorem \ref{thm5} set
$$
\mathbb S=  \sigma(A)\sigma_{BW}(B)\cup\sigma_{BW}(A)\sigma (B).
$$

\begin{thm}\label{thm11}Let $\X$ and $\Y $ be two Banach spaces and consider
$A\in B(\X)$  and $B\in B(\Y)$ such that $A$ is an algebraic but not nilpotent operator and $B$ is a non-algebraic operator that satisfies the generalized
Browder's theorem. Then, if $A\otimes B\in B( \X\overline{\otimes}\Y)$ (respectively if $\ST\in B(B(\Y, \X))$) satisfies the generalized Browder's theorem,
the following statements are equivalent.\par
\noindent \rm (i) \it The B-Weyl spectrum inclusion for $A\otimes B$ (respectively for $\ST$) holds;\par
\noindent \rm (ii) \it $B$ is not Drazin invertible.\par
\noindent Furthermore, if one of the equivalent statements holds, then $\mathbb S=\sigma_{BW}(A\otimes B)$ (respectively
$\mathbb S=\sigma_{BW}(\ST))$,
while if this is not the case, then $\mathbb S=\sigma_{BW}(A\otimes B)\cup \{0\}$ (respectively
$\mathbb S=\sigma_{BW}(\ST)\cup \{0\}$).\par
\end{thm}
\begin{proof} Note that since according to Remark \ref{rem6}(ii) $\sigma (A)=\Pi(A)$, $\mathbb S=\Pi(A)\sigma_{BW}(B)$.
On the other hand, recall that $\sigma_{BW}(B)= I( B)\cup $ acc $\sigma  (B)$ and
$\sigma_{BW}(A\otimes B)= I(A\otimes B)\cup $ acc $\sigma (A\otimes B)$
 (\cite[Theorem 12]{Bo}). In particular,
since $I(A\otimes B)\setminus\{0\}= (\Pi(A)\setminus
\{0\})(I(B)\setminus\{0\}$) (Theorem \ref{thm9}) and acc
$\sigma(A\otimes B)\setminus\{0\}= (\Pi(A))\setminus\{0\})(\hbox{\rm
acc } \sigma(B)\setminus\{0\})$   (\cite[Theorem 6]{HK}), $\mathbb
S\setminus\{0\}=\sigma_{BW} (A\otimes B)\setminus\{0\}$. Therefore,
according to Lemma \ref{lem16}, statement (i) is equivalent to the
following implication: $0\in \mathbb S\Rightarrow 0\in
\sigma_{BW}(A\otimes B)$. Now, $0\in \mathbb S$ if and only if
$0\in \Pi(A)$ or $0\in \sigma_{BW}(B)$. If $0\in  \sigma_{BW}(B)$,
then using in particular statements (ii), (iv) and (vi) of Theorem
\ref{thm8} and the fact that $A$ is not nilpotent, it is not
difficult to prove that $0\in\sigma_{BW}(A\otimes B)$. On the other
hand, if $0\in \Pi(A)$, according to what has been proved, if
$0\in\sigma_{BW}(B)$, then $0\in\sigma_{BW}(A\otimes B)$, while if
$0\in\Pi (B)$ or $0\notin \sigma (B)$, equivalently if $B$ is Drazin
invertible, then according to statements (iii) and (v) of Theorem
\ref{thm8}, $0\in\Pi (A\otimes B) =\sigma (A\otimes
B)\setminus\sigma_{BW}(A\otimes B)$. Consequently, since all the
possible cases have been considered, the B-Weyl spectrum inclusion
for $A\otimes B$ holds if and only if $B$ is not Drazin invertible.
\par \indent The last statement is clear.\par \indent The statements
concerning the operator $\ST \in B(B(\Y, \X))$ can be proved in a
similar way.
\end{proof}

\indent Naturally, under the same conditions as Theorem \ref{thm11},
if the properties of $A$ and $B$ are interchanged, similar
statements can be proved. Next follows the remaining case, i.e.,
when both operators are not algebraic.\par

\begin{thm}\label{thm38}Let $\X$ and $\Y $ be two Banach spaces and consider
$A\in B(\X)$ and $B\in B(\Y)$ two non-algebraic operators that satisfy the generalized Browder's theorem.
Then, if $A\otimes B\in B( \X\overline{\otimes}\Y)$ (respectively if $\ST\in B(B(\Y, \X))$) satisfies the generalized Browder's theorem,
the following statements are equivalent.\par
\noindent \rm (i) \it The B-Weyl spectrum inclusion for $A\otimes B$ (respectively for $\ST$) holds;\par
\noindent \rm (ii)\it  $0\notin\Pi (A\otimes B)$ $(=\Pi (\ST ))$;\par
\noindent \rm (iii)\it $A\otimes  B$ (respectively $\ST$) is invertible or $A\otimes B$ (respectively $\ST$)
is not Drazin invertible.
\noindent Furthermore, if one of the equivalent statements holds, then $\mathbb S=\sigma_{BW}(A\otimes B)$ (respectively
$\mathbb S=\sigma_{BW}(\ST))$,
while if this is not the case, then $\mathbb S=\sigma_{BW}(A\otimes B)\cup \{0\}$ (respectively
$\mathbb S=\sigma_{BW}(\ST)\cup \{0\}$).\par
\end{thm}

\begin{proof} Consider  the operator $A\otimes B\in B(\X\overline{\otimes} \Y)$.
Recall that since  $A$, $B$ and $A\otimes B$ satisfy the generalized
Browder's theorem, according to \cite[Theorem 12]{Bo},
$\sigma_{BW}(A)= I(A)\cup $ acc $\sigma (A)$, $\sigma_{BW}(B)=
I(B)\cup $ acc $\sigma (B)$ and $\sigma_{BW}(A\otimes B)= I(A\otimes
B)\cup $ acc $\sigma (A\otimes B)$. Now set $\mathbb A=$ (acc
$\sigma (A))\sigma (B)\cup \sigma (A)$(acc $\sigma (B))$ and
$\mathbb B =I(A)I(B)\cup I(A)\Pi (B)\cup \Pi (A)I(B)$. Note that
$\mathbb B\setminus \{0\}=\mathbb L=I(A\otimes B)\setminus \{0\}$
(Theorem \ref{thm9}) and $\mathbb S= \mathbb A \cup \mathbb B$. \par
\noindent (i) $\Rightarrow$ (ii). Suppose that $0\in\Pi (A\otimes
B)\subseteq\sigma(A\otimes B)=\sigma(A)\sigma(B)$. Then, since
neither $A$ nor $B$ is algebraic, $0\in\mathbb
S\subseteq\sigma_{BW}(A\otimes B)$, which is impossible for
$\Pi(A\otimes B)=\sigma(A\otimes B)\setminus\sigma_{BW}(A\otimes
B)$.\par \noindent (ii) $\Rightarrow$ (iii). Apply \cite[Theorem
12]{Bo}.\par \noindent (iii) $\Rightarrow$ (i). Note that three cases 
must be considered: $0\notin \sigma (A\otimes B)$;  $0\in$
acc $\sigma(A\otimes B)$;  $0\in I(A\otimes B)$.
Suppose that
$0\notin \sigma (A\otimes B)$ or $0\in$ acc $\sigma(A\otimes B)$.
Then, according
to \cite[Theorem 6]{HK}, acc $\sigma(A\otimes B)=\mathbb A$. If
$0\notin \sigma(A\otimes B)$, then $\mathbb B=I(A\otimes B)$, in
particular, $\mathbb S= \sigma_{BW}(A\otimes B)$, while if $0\in$
acc $\sigma (A\otimes B)=\mathbb A$, since $\mathbb B\setminus
\{0\}=I(A\otimes B)$, $\mathbb S= \sigma_{BW}(A\otimes B)$.\par
\indent  Next
suppose that $0\in I(A\otimes B)\subseteq \sigma_{BW}(A\otimes
B)\subseteq \sigma(A\otimes B)$. Since neither $A$ nor $B$ is
algebraic, $0\in\mathbb S$. Moreover, since $\mathbb
B\setminus\{0\}=I(A\otimes B)\setminus\{0\}$ and acc
$\sigma(A\otimes B)=\mathbb A\setminus\{0\}$ (\cite[Theorem 6]{HK}),
$\mathbb S=\sigma_{BW}(A\otimes B)$. \par \indent Concerning the
last statement, according to Lemma \ref{lem16}, it is enough to
consider the case $\sigma_{BW}(A\otimes B)\subsetneq\mathbb S$.
Suppose that $0\in \Pi (A\otimes B)$. Then, since neither $A$ nor
$B$ is algebraic, $0\in\mathbb S\setminus\sigma_{BW}(A\otimes B)$.
However, since acc $\sigma(A\otimes B)=\mathbb A\setminus \{0\}$
(\cite[Theorem 6]{HK}) and $I(A\otimes B)\setminus\{0\}=\mathbb
B\setminus \{0\}$,
 $\mathbb S=\sigma_{BW}(A\otimes B)\cup \{0\}$.\par
\indent Finally, a similar argument proves the statements concerning the left-right multiplication operator.
\end{proof}

\begin{rema}\label{rem16} \rm Note that under the same hypotheses as Theorems \ref{thm11} and \ref{thm38},
if $\sigma_{BW}(A\otimes B)\subsetneq  \sigma(A)\sigma_{BW}(B)\cup\sigma_{BW}(A)\sigma (B)$,
then $\sigma(A)\sigma_{BW}(B)\cup\sigma_{BW}(A)\sigma (B)=\sigma_{BW}(A\otimes B)\cup\{0\}$.
Moreover, a similar observation holds for $\ST$.
\end{rema}

\indent In the following theorem the  transfer property for the
generalized Browder's theorem will be characterized. Note that if an
operator is not Drazin invertible, then it is not algebraic. Recall
that according to \cite[Theorem 2.1]{AZ}, Browder's theorem and the
generalized Browder's theorem are equivalent. Moreover, recall that
Browder's theorem both for the tensor product operator and for  the
elementary operator is equivalent to the respective Weyl spectrum
equality, see \cite[Theorem 3]{DDK} and \cite[Theorem 4.5]{BDJ}
respectively.\par

\begin{thm}\label{thm14}Let $\X$ and $\Y $ be two Banach spaces and consider
$A\in B(\X)$ and $B\in B(\Y)$ two operators  that satisfy the
generalized Browder's theorem. Suppose either that $A$ and $B$ are
not algebraic  and $0\notin \Pi (A\otimes B)$ $(= \Pi(\ST))$ or that
only one of them, say $A$, is algebraic but not nilpotent and the
other, say $B$, is not Drazin invertible.\par \noindent (a). The
following statements are equivalent.\par \noindent \rm  (i)\it
The (generalized) Browder's theorem for $A\otimes B\in
B(\X\overline{\otimes}\Y)$ holds;\par \noindent \rm (ii)\it
$\sigma_{BW}(A\otimes B)=\sigma (A)\sigma_{BW}( B)\cup
\sigma_{BW}(A)\sigma(B)$.\par \noindent \rm (iii)\it
$\sigma_w(A\otimes B)= \sigma_w(A)\sigma(B)\cup
\sigma(A)\sigma_w(B)$.\par \noindent (b). The following statements
are equivalent.\par \noindent \rm (i)\it The (generalized) Browder's
theorem for $ \ST\in B(B(\Y,\X))$ holds;\par \noindent \rm (ii)\it
$\sigma_{BW}(\ST)=\sigma (A)\sigma_{BW}( B)\cup
\sigma_{BW}(A)\sigma(B)$.\par \noindent \rm (iii)\it $\sigma_w(\ST)=
\sigma_w(A)\sigma(B)\cup \sigma(A)\sigma_w(B)$.\par
\end{thm}
 \begin{proof} Apply Theorems \ref{thm5}, \ref{thm11} and \ref{thm38}.
 \end{proof}

\vskip.3truecm
\noindent Enrico Boasso\par
\noindent E-mail address: enrico\_odisseo@yahoo.it\par
\vskip.3truecm
\noindent B. P. Duggal\par
\noindent 8 Redwood Grove, Northfield Avenue,\par
\noindent Ealing, London W5 4SZ,  United Kingdom\par
\noindent E-mail address: duggalbp@gmail.com

\end{document}